\documentclass[runningheads,11pt]{llncs}

\usepackage{amssymb}
\usepackage{graphicx}
\usepackage{mathrsfs}
\usepackage{hyperref}
\usepackage{color}
\usepackage{amsmath}
\usepackage{ulem}
\usepackage{url}
\urldef{\mailsc}\path|slou1@luc.edu|    
\newcommand{\keywords}[1]{\par\addvspace\baselineskip
\noindent\keywordname\enspace\ignorespaces#1}

\numberwithin{equation}{section}

\def\sE{{\mathscr E}}
\def\sF{{\mathscr F}}
\def\sC{{\mathscr C}}
\def\cB{{\mathcal{B}}}
\def\cE{{\mathcal{E}}}

\def\bR{{\mathbb{R}}}

\def\ts{{\mathtt{s}}}

\def\re{{\mathrm{e}}}
\def\fm{{\mathfrak{m}}}

\def\b0{{\textbf{0}}}

\def\wh{\widehat}
\def\wt{\widetilde}
\def\<{\langle}
\def\>{\rangle}

\newcommand{\IE}{{\mathbb{E}}}
\newcommand{\IP}{{\mathbb{P}}}
\newcommand{\IR}{{\mathbb{R}}}

\begin{document}


\title{Explicit heat kernels of a model of distorted Brownian motion on spaces with varying dimension}

\titlerunning{Heat kernels for dBM with varying dimension}

%
%
\author{ Shuwen Lou%
}
\authorrunning{ S. Lou}

\institute{Loyola University Chicago\\ Chicago, IL 60660, USA.\\
\mailsc\\
}

%
%

\maketitle

\begin{abstract}
In this paper, we study a particular model of distorted Brownian motion (dBM) on  state spaces with varying dimension. Roughly speaking, the state space of such a process consists of two components: a $3$-dimensional component and a $1$-dimensional component. These two parts are joined together at the origin. The restriction of dBM on the $3$- or $1$-dimensional  component  receives a strong ``push" towards the origin. On each component, the ``magnitude" of the ``push" can be parametrized by a constant $\gamma >0$. In this article, using probabilistic method, we get the exact expressions for the  transition density functions of dBM with varying dimension for any $0<t<\infty$. 

\end{abstract}

\section{Introduction}\label{Sec-intro}

The concept of $3$-dimensional distorted Brownian motion arises in statistical physics. To give  a brief description to $3$-dimensional dBM, we consider the the standard $3$-dimensional Brownian motion on the path space denoted by $(\Omega, \{\IP^x\}_{x\in \IR^3}, \omega(t), t\ge 0)$. For the Hamiltonian we select $H(\omega)=\int_0^t \mathbf{1}_{\{|x|\le \epsilon\}}(\omega (s))ds $. For  $A(\epsilon)=\frac{\pi^2}{8\epsilon^2}+\frac{\gamma}{\epsilon}$ where $\gamma$ is a positive parameter, we define the Gibbs measure $\IP^x_{\beta, t}$ by setting 
\begin{equation*}
\frac{d\IP^x_{\beta, t}}{d\IP^x}=\frac{\exp\left\{ A(\epsilon)\int_0^t \mathbf{1}_{\{|x|\le \epsilon\}}(\omega (s))ds\right\}}{Z_{\beta, t}(x)},
\end{equation*}
where 
\begin{equation*}
Z_{\beta, t}(x)=\IE^x \left[\exp\left\{A(\epsilon) \int_0^t \mathbf{1}_{\{|x|\le \epsilon\}}(\omega (s))ds\right\}\right]
\end{equation*}
is the normalizing constant making $\IP^x_{\beta, t}$ a probability measure. This model arises from the discrete homopolymer model: The latter is similar to the model described above, with the only changes being that $3$-dimensional Brownian path $\omega(t)$ is replaced with $3$-dimensional continuous time simple random walk on $\mathbb{Z}^3$, $\mathbf{1}_{\{|x|\le \epsilon\}}$ is replaced with $\delta_0$, and that $A(\epsilon)$ is replaced with $\gamma$. 

For the continuous model on $\IR^3$ we introduce above, as $\epsilon\rightarrow 0$, there is a limiting process associated with it. The rigorous meaning of the ``limit" can be found in \cite{CM}. Roughly speaking, as $\epsilon\rightarrow 0$, the resolvents converge to another family of resolvents which has a Markov process associated with it. We call such a  limiting process  a $3$-dimensional distorted Brownian motion with parameter $\gamma$.

Many interesting  properties of $3$-dimensional dBM  have been investigated in \cite{CM}, \cite{CKMV1}, and \cite{CKMV2}, including its explicit transition densities and behaviors near the origin. Later in \cite{FL17}, Fitzsimmons and Li give a very thorough description to  this process by means  of  its associated Dirichlet form.

Unlike a $3$-dimensional standard Brownian motion which does not hit any  singleton, a $3$-dimensional dBM is subject to a strong push towards the origin. Therefore, it  is recurrent and has positive capacity at the origin, which allows us to study  such a process on a state space with varying dimension.    Such a  state space with varying dimension consists of two components: a $3$-dimensional component and a $1$-dimensional component. These two parts are joined together at the origin.The study of Markov processes with varying dimension was originated in \cite{CL}, where the model is constructed by joining together a $2$-dimensional Brownian motion and a $1$-dimensional Brownian motion on half line. Since $2$-dimensional Brownian motion does not hit any singleton, the construction of such a process  with varying dimension utilizes the method of ``darning", i.e.,  setting the resistance on a $2$-dimensional disc equal to zero. The disc is  centered at the intersection of the plane and the pole. The model studied in \cite{CL} is a toy model of Markov processes with varying dimension, but many important properties  as well as techniques of analyzing such process have been developed in that article.

In this paper, we first give a more precise description to dBM with varying dimension. The state space of such a process is embedded in $\IR^4$. We let $\mathbb{R}^4\supset E_1:=\{(x, 0): x\in \mathbb{R}^3\}\cong \mathbb{R}^3$ and $\mathbb{R}^4\supset E_2:=\{(0,0,0, x): x\in [0, +\infty)\}\cong [0, +\infty)$. Set
\[
	E:=E_1\cup E_2. 
\]
Clearly, $E_1\cap E_2=(0,0,0,0)=:\b0\in \mathbb{R}^4$. $E$ is a topological space and a  neighborhood of $\b0$ defined as $\{\b0\}\cup \left(V_1\cap E_1 \right)\cup \left(V_2\cap E_2\right)$ for some neighborhood $V_1$ of $\b0$ in $E_1$ and $V_2$ of $\b0 $ in $E_2$.  The restriction of dBM with varying dimension on $E_1$ and $E_2$ behaves like a $3$-dimensional and an $1$-dimensional distorted Brownian motion, respectively.  We emphasize that in this paper, except for $\b0\in \IR^4$, both  vectors and  scalars are unbolded. For example, we may use $``x"$ to denote an element in $E_1\subset \IR^4$.

The main result of this paper is obtaining the explicit expression for the transition density function of  distorted Brownian motion with varying dimension for all $t>0$, for the case that the ``parameter of distortion" $\gamma>0$ is the same on both the $3$-dimensional and $1$-dimensional components. The key observation is that for this case, the signed radial process of this process is symmetric about $0$. Therefore the ``absolute"  radial process is actually  a Brownian motion reflected at zero with a constant drift pushing towards the origin.  From here, realizing  that the distribution of the signed radial process can be ``decomposed" into a  Brownian motion with drift reflected at the origin and a Brownian motion with drift killed at the origin, both $1$-dimensional, we derive the explicit global transition density of the process for all $t>0$.

Before we state the main results, we introduce the  underlying measure and the metric  on the state space.    Throughout this paper, we denote by   $|x-y|$  the Euclidean distance between $x$ and $y$ if either $x,y\in E_1$ or $x,y\in E_2$. This can either be viewed as Euclidean distance on $\IR^4$, or its projection onto $\IR_+$ or $\IR^3$.  By slightly abusing the notation,  we let
\begin{equation}\label{def-metric}
|x-y|:=|x-\b0|+|y-\b0|, \quad \text{if     }x\in E_1,\, y\in E_2.
\end{equation} 
Fix a parameter $\gamma>0$.  The measure $m_\gamma$ on $E$ is given as 
\begin{eqnarray}\label{definition-m}
	m_\gamma(dx):=\left\lbrace
	\begin{aligned}
		&\frac{\gamma}{2\pi}\frac{e^{-2\gamma |x|}}{|x|^2}d{x}, \quad \text{ on }E_1, \\
		&2\gamma e^{-2\gamma |x|}dx,\quad \text{ on }E_2. 
	\end{aligned} \right.
\end{eqnarray}
In the above $dx$  is  Lebesgue measure on $\IR^4$. $m_\gamma$ is well-defined because $0$ is of zero-Lebesgue-measure for both $1$-dimensional and $3$-dimensional spaces.

 In this  paper, we denote the  main process of interest, the distorted Brownian motion with varying dimension by $M$, whose rigorous definition is given in  Definition \ref{definition-dBMVD}, Section 2.2.  For any connected $C^{1,1}$ open subset $D$ of $E$, we let $M^D$ be the part process of dBM with varying dimension killed upon exiting $D$, and  denote by $p_{D}(t,x,y)$ its  transition density function. Similar notations can be defined for other stochastic processes.  Throughout this paper, we set
\begin{equation*}
\IR_+:=(0, +\infty), \quad \IR_-:= (-\infty, 0).
\end{equation*}
 Given a subset $A\subset E$, we define $\sigma_A:=\inf\{t>0, M_t\in A\}$ and  $\tau_{_A}:=\inf\{t\geq 0:M _t \notin A\}$. By convention $\inf \emptyset:=\infty$.   Similar notations are defined for other stochastic processes as well.  For notation convenience, in this paper, given $x\notin K$, we set
\begin{equation}\label{e:1.6}
p(t,x,y;\; \sigma_K\le  t):=p(t,x,y)-p_{E\backslash K}(t,x,y), 
\end{equation}
where $K$ is a compact subset of $E$ and $p_{E\backslash K}(t,x,y)$ is the transition density of the part process killed upon exiting $E\backslash K$.
In other words, for any non-negative function $f\geq 0$ on $E$,
\begin{equation}\label{e:1.7}
 \int_E   p(t,x,y; \sigma_K<t)  f(y) m_\gamma(dy) = \IE_x \left[ f(M_t); t\ge  \tau_{E\backslash K} \right] = \IE_x \left[ f(M_t); t\ge  \sigma_{K} \right].
\end{equation}
Thus while $p_{E\backslash K}(t,x,y)$ gives the probability density
that $M$ starting from $x$ hits $y$ at time $t$ without visiting  $K$, 
 $p(t,x,y; \sigma_K<t) $ is the probability density for $M$ starting from $x$ visits $K$ before ending up at $y$ at time $t$. Similar notations are defined for other stochastic processess. In this paper, in order to distinguish the transition densities or stopping times of different processes, sometimes we use superscripts, e.g., $p^M(t,x,y)$ or $\sigma^M$, to emphasize that it is the transition density  or hitting time of process $M$, whenever there might be confusion.

 The main result of this paper is the following explicit transition density function for dBM with varying dimension.
 \begin{theorem}\label{main-thm}
Fix $\gamma>0$. With respect to the measure $m_\gamma$ given in \eqref{definition-m}, for all $t>0$, the transition density of $M$, denoted by $p(t,x,y)$, has the following expression:
\begin{description}
\item{(i)}
\begin{equation*}
p(t,x,y)=q(t,x, y)+\frac{1}{2}\left(p^{\wh{Y}}(t, |x|, |y|)-p^Y_{\IR_+}(t, |x|, |y|)\right), \quad  x,y\in E_1\backslash \{\b0\};
\end{equation*}
\item{(ii)} 
\begin{equation*}
p(t,x,y)=\frac{1}{2}\left(p_{E_2\backslash \{\b0\}}(t,x,y)+p^{\wh{Y}}(t, |x|, |y|)\right), \quad x,y\in E_2\backslash \{\b0\};
\end{equation*}
\item{(iii)} 
\begin{equation*}
p(t,x,y)=\frac{1}{2}\left(p^{\wh{Y}}(t, |x|, |y|)-p^Y_{\IR_+}(t, |x|, |y|)\right), \quad x\in E_1\backslash \{\b0\}, y\in E_2\backslash \{\b0\},
\end{equation*}
\item{(iv)}
\begin{equation*}
p(t,\b0,y)=\frac{1}{2}\left(1+\frac{1}{\pi \gamma} e^{\gamma |y|-\gamma t^2/2} \int_0^\infty \frac{s e^{-s^2t/2}}{s^2+\gamma^2} \left(s^2\cos(s|y|)-s\gamma \sin(s|y|)\right)ds \right),\quad y\in E.
\end{equation*}
\end{description}
where the explicit expressions of $q(t,x,y)$, $p^{\wh{Y}}(t,x,y)$, $p^Y_{\IR_+}(t, |x|, |y|)$,  and $p_{E_2\backslash \{\b0\}}(t,x,y)$ are given in \eqref{EQQTX}, \eqref{density-RBM-drift}, \eqref{density-pY_R-},  and  \eqref{density-p_E2} respectively.
\end{theorem}
\begin{remark}
\begin{description}
\item{(i)} $q(t,x,y)$ denotes the transition density  of the part process of $M$  on $E_1\backslash \{\b0\}$ killed upon hitting $\b0$, with respect to $m_\gamma$.
\item{(ii)} $Y$  is the signed radial process  of $M$, defined at the beginning of Section \ref{S:3.2}. $p^Y(t,x,y)$ is the density of $Y$  with respect to the measure  $\wt{\mathfrak{m}}$ characterized in \eqref{def-m-tilde}. In fact, $\wt{\mathfrak{m}}$ is the symmetrizing  measure  for   $Y$.
\item{(iii)}  $\wh{Y}:=|M| = |Y|$, i.e., the ``unsigned" radial process.  With respect to the measure  $m^{(3)}$, $p^{\wh{Y}}(t,x,y)$ is the density of $\wh{Y}$.
\item{(iv)} $p_{E_2\backslash \{\b0\}}(t,x,y)$ is the density of the part process of $M$ on $E_2$  with respect to $m_\gamma^2$.
\item{(v)} Since $p(t, x, y)$ is symmetric in $(x, y)$ with respect to  $m_\gamma$, the three cases (i)-(iii) essentially  cover all the cases for $x,y\in E$.
\end{description}
\end{remark}

The rest of this paper is organized as follows: In Section \ref{SEC2-1} we give a background on 3-dimensional distorted Brownian motion. Most of the results in Section \ref{SEC2-2} can be found in \cite{FL17}. In Section \ref{SEC2-2}, we introduce the definition of dBM with varying dimension as well as its Dirichlet form characterizations. In Section \ref{SEC3-1}, we show that for dBM with varying dimension,  the intersection of the two components $E_1$ and $E_2$, $\b0$ has positive capacity. Furthermore, $\b0$ can be visited in finte times with probability one starting from  everywhere. The signed radial process of dBM with varying dimension is characterized by an SDE  in Section \ref{S:3.2}. As a  quick corollary, the signed radial process is symmetric about 0. Finally, the proof to Theorem \ref{main-thm} is presented in Section \ref{Proof-HK}. For the purpose of  better organizing this proof, we present the result of each of (i)-(iv) in Theorem \ref{main-thm} as a separate proposition in Section \ref{Proof-HK}.

\section{Preliminary}\label{Sec-prelim}

In this section, we give an introductory overview on distorted Brownian motion and dBM with varying dimension. Most the results in this section can be found in \cite{FL17}.

\subsection{$3$-dimensional distorted BM}\label{SEC2-1}

This is a summary of \cite{FL17}. Fix 
\begin{equation}\label{EQ1PGX}
\psi_\gamma(x):=\sqrt{\frac{\gamma}{2\pi}}\cdot \frac{\mathrm{e}^{-\gamma|x|}}{|x|},\quad x\in \mathbb{R}^3.
\end{equation}
Note that $\int \psi_\gamma(x)^2dx=1$. Set the following measure on $\IR^3$: 
 \begin{equation}\label{def-m3}
 \mathfrak{m}^{(3)}(dx):=\psi_\gamma(x)^2dx.
 \end{equation}
 We define an energy form on $L^2(\mathbb{R}^3,\mathfrak{m}^{(3)})$ as follows:
\begin{equation*}
\left\{
\begin{aligned}
&\mathcal{F}^{(3)}:=\left\{f\in L^2(\mathbb{R}^3,m^{(3)}): \nabla f\in L^2(\mathbb{R}^3,\mathfrak{m}^{(3)})\right\}, \\
&\mathcal{E}^{(3)}(f,g):=\frac{1}{2}\int_{\mathbb{R}^3}\nabla f(x)\cdot \nabla g(x)\mathfrak{m}^{(3)}(dx),\quad f,g\in \mathcal{F}^{(3)}. 
\end{aligned}
\right.
\end{equation*}
The next theorem includes  some facts about $(\mathcal{E}^{(3)},\mathcal{F}^{(3)})$, in which (ii)  is critical for us to construct dBM with varying dimension. 

\begin{theorem}[Cf. \cite{FL17}]\label{THM1}
The following statements hold:
\begin{itemize}
\item[(i)] $(\mathcal{E}^{(3)},\mathcal{F}^{(3)})$ is a regular Dirichlet form on $L^2(\mathbb{R}^3, \mathfrak{m}^{(3)})$ with $C_c^\infty(\mathbb{R}^3)$ being its special standard core. Denote its associated Markov process by $X^{(3)}=(X^{(3)}_t)_{t\ge 0}$. 
\item[(ii)] Any singleton $x\neq 0$ in $\IR^3$ is $\mathcal{E}^{(3)}$-polar, but $0\in \mathbb{R}^3$ is of positive capacity. 
\item[(iii)] $(\mathcal{E}^{(3)},\mathcal{F}^{(3)})$ or $X^{(3)}$ is recurrent, conservative and irreducible. Particularly, $\mathfrak{m}^{(3)}$ is an invariant measure of $X^{(3)}$. 
\end{itemize}
\end{theorem}

The diffusion $X^{(3)}$ is called a $3$-dimensional  distorted Brownian motion. In the following we provide more detailed description to it. The third assertion of Theorem~\ref{THM1} states that $X^{(3)}$ is irreducible recurrent. This implies (by \cite[Theorem~4.7.1]{FOT})
\[
	\IP_x\left(\sigma_{\{0\}}<\infty\right)=1,\quad \text{for q.e. }x\in \IR^3. 
\]
Particularly, $\IP_0\left(\sigma_{\{0\}}<\infty\right)=1$. In fact, we have  $\IP_0\left(\sigma_{\{0\}}=0\right)=1$.  This means $0$ is a regular point. 
Heuristically speaking, $X^{(3)}$ behaves like a one-dimensional Brownian motion near $0$. 

Next we give some remarks on the rotational invariance of $X^{(3)}$, some of which will be used later in this paper. These results can be found in \cite[Section 3]{FL17}.
\begin{description}
\item{(i)} $X^{(3)}$ is isotropic in the sense that if $T:\mathbb{R}^3\rightarrow \mathbb{R}^3$ is isotropic (i.e. $T: x=(r, \theta, \varphi)\mapsto (r, \theta+\theta_0, \varphi+\varphi_0)$ for some given $\theta_0$ and $\varphi_0$), then $X^{(3)}$ and $T(X^{(3)})$ are equivalent (i.e. they share the same Dirichlet form). 
\item{(ii)} It holds
\[
	X^{(3)}_t=(r_t, \vartheta_{A_t}),\quad t\geq 0, 
\]
where $r_t$ is a diffusion on $[0,\infty)$, $\vartheta$ is a BM on $S^2$ and $A_t$ is a PCAF of $(r_t)$. The Revuz   measure of $A$ is $\mu_A(du)=l(du)/u^2$ ($l$ is given in  \eqref{EQ1LUG}). 
\item{(iii)} The radial part $r_t$ of $X^{(3)}$ is a diffusion reflected at $0$ with the scale function 
\[
	\ts(u)=\frac{1}{4\gamma^2}\mathrm{e}^{2\gamma u},\quad u\in [0,\infty)
\]
and speed measure 
\begin{equation}\label{EQ1LUG}
l(du)=2\gamma\mathrm{e}^{-2\gamma u}du. 
\end{equation}
\item{(iv)} $r_t$ satisfies
\begin{equation}\label{radial-3d-dbm}
	r_t-r_0=\beta_t-\gamma t+\pi \gamma \cdot \mathfrak{l}^0_t,
\end{equation}
where $\beta_t$ is a one-dimensional Brownian motion  and $\mathfrak{l}^0$ is the local time of $r$ in the sense of Revuz measure at $0$. 
\end{description}

\subsection{Distorted Brownian motion  on spaces with varying dimension}\label{SEC2-2}

In this subsection, we rigorously give the definition for  distorted Brownian motion with varying dimension $M$ on the state space $E$. Recall that we have mentioned in Section \ref{Sec-intro} that $E=E_1\cup E_2$. The restriction of dBM with varying dimension $M$ on $E_1$ is induced by  the $3$-dimensional distorted Brownian motion $X^{(3)}$ defined  in Section \ref{SEC2-1}. 
Set the inclusion 
\[
	\iota_1: \mathbb{R}^3\rightarrow E_1, \quad x\mapsto (x, 0). 
\]
Then we define $m^1_\gamma:=\mathfrak{m}^{(3)}\circ \iota_1^{-1}$ as a measure on $E_1$, where $\mathfrak{m}^{(3)}$ is defined in \eqref{def-m3}. Thus $\iota_1(X^{(3)})$ is a distorted Brownian motion on $E_1$ associated with the following  Dirichlet form $(\sF^1, \sE^1)$ on $L^2(E_1,m^1_\gamma)$.
\begin{equation*}
\left\{
\begin{aligned}
	&\sF^1:=\{f: f\circ \iota_1\in \mathcal{F}^{(3)}\}, \\
	&\sE^1(f,g):=\mathcal{E}^{(3)}(f\circ \iota_1, g\circ \iota_1),\quad f,g\in \sF^1. 
\end{aligned}
\right.
\end{equation*}
To introduce the behavior of $M$ on the one-dimensional part $E_2$, for  $\gamma >0$ we first set
\begin{equation}\label{def-phi_gamma}
\phi_\gamma (u):=\sqrt{2\gamma} e^{-\gamma u}, \quad \text{for } u\in \IR_+\cup \{0\}.
\end{equation}
Now define the following measure on $\IR_+\cup \{0\}$:
$$
\mathfrak{m}^{(+)}(du):=\phi_\gamma(u)^2du=2\gamma\mathrm{e}^{-2\gamma u}du.$$
We now consider the following Dirichlet form on $L^2(\mathbb{R}_+\cup \{0\}, \mathfrak{m}^{(+)})$:
\begin{equation}\label{def-EE+}
\left\{
\begin{aligned}
	&\mathcal{F}^{(+)}:=\left\{f\in L^2(\mathbb{R}_+\cup \{0\}, \mathfrak{m}^{(+)}):  f'\in L^2(\mathbb{R}_+,\mathfrak{m}^{(+)}),\, f \text{ absolutely continuous on }[0, +\infty)\right\},  \\
	&\mathcal{E}^{(+)}(f,g):=\frac{1}{2}\int_{\mathbb{R}_+} \nabla f(u)\nabla g(u)\mathfrak{m}^{(+)}(du),\quad f,g\in \mathcal{F}^2,
\end{aligned}
\right.
\end{equation}
Denote the diffusion associated with $(\mathcal{F}^{(+)}, \mathcal{E}^{(+)})$  by $X^{(+)}$.  
Set the inclusion
\[
	\iota_2: \mathbb{R}_+\cup \{0\}\rightarrow E_2,\quad u\mapsto (0,0,0, u). 
\]
Then $\iota_2(X^{(+)})$ is a diffusion on $E_2$ associated with the following Dirichlet form  $(\sF^2, \sE^2)$ on $L^2(E_2, m^2_\gamma)$ where $m^2_\gamma:=\mathfrak{m}^{(+)}\circ \iota_2^{-1}$:
\begin{equation*}
\left\{
\begin{aligned}
	&\sF^2:=\{f: f\circ \iota_2\in \mathcal{F}^{(+)}\}, \\
	&\sE^2(f,g):=\mathcal{E}^{(+)}(f\circ \iota_2, g\circ \iota_2),\quad f,g\in \sF^2.
\end{aligned}
\right.
\end{equation*}
Now we are ready to introduce the defintion of dBM with varying dimension, as well as its associated Dirichlet form. We  note that by the definition of $m_\gamma$ in \eqref{definition-m}, it is easy to see that $m_\gamma|_{E_i}=m_\gamma^i$, for $i=1,2$. The definition and explanation of ``quasi-continuous" can be found, e.g., in \cite[Definition 1.2.12, Theorem 2.3.4]{CF}.
\begin{proposition}\label{THM2}
Let 
\begin{equation}\label{dirichletform-dbm-1}
\left\{
\begin{aligned}
&\sF:=\left\{ f\in L^2(E,m_\gamma): f|_{E_1}\in \sF^1, f|_{E_2}\in \sF^2, \widetilde{f|_{E_1}}(\b0)=\widetilde{f|_{E_2}}(\b0) \right\}, 
\\
&\sE(f,g):=\sE^1(f|_{E_1}, g|_{E_1})+\sE^2(f|_{E_2}, g|_{E_2}),\quad f,g\in \sF. 
\end{aligned}
\right.
\end{equation}
In \eqref{dirichletform-dbm-1}, $\widetilde{f|_{E_i}}$ is the $\sE^i$-quasi-continuous version of $f|_{E_i}$. Then $(\sE,\sF)$ is a  strongly local regular Dirichlet form on $L^2(E,m_\gamma)$. Therefore there exists a unique diffusion process  associated with it.
\end{proposition}
\begin{proof}
Clearly, $(\sE,\sF)$ is a symmetric bilinear form satisfying  Markovian property. The strong locality of $(\sE,\sF)$ is indicated by that of $(\sE^1,\sF^1)$ and $(\sE^2,\sF^2)$.  It remains to prove the regularity of $(\sE,\sF)$. Take $\sC^1:=C_c^\infty(\bR^3)\circ \iota_1^{-1}$ and   $\sC^2:=C_c^\infty(\bR_+)\circ \iota_2^{-1}$. In view of \cite[Theorem 2.1]{FL17}, we know $\sC^i$ is a special standard core of $(\sE^i, \sF^i)$, for $i=1,2$.   Set
\begin{equation}\label{EQ3CFF}
	\mathscr{C}:=\{f\in \sF: \,  f|_{E_1}\in \sC^1,\, f|_{E_2}\in  \sC^2\}.
\end{equation}
On account of \cite[Lemma 1.3.12]{CF}, to prove the regularity of $(\sE, \sF)$, it suffices to show $\sC$ is dense in $C_c(E)$ relative to the uniform norm and dense in $\sF$ relative to the $\sE_1$-norm. $\sC$ is clearly an algebra, i.e. $f,g\in \sC$ implies $c_1\cdot f+c_2 \cdot g, f\cdot g\in \sC$ for any constants $c_1,c_2$. To show $\sC$ separates points in $E$, without loss of generality, we consider $x\in E_2,\, y\in E_1\setminus \{\b0\}$. Since $\sC^2$ is a special standard core of $(\sE^2,\sF^2)$, there exists a function $f_2\in \sC^2$ such that $f_2(\b0)=f_2(x)=1$. Another function $f_1\in \sC^1$ can be taken to separate $\b0$ and $y$, i.e. $f_1(\b0)\neq f_1(y)$. Define a function $f$ on $E$ by
\[
 f|_{E_1}:=f_1, \quad f|_{E_2}:= f_1(\b0)\cdot f_2.
\]
Then $f\in \sC$ and $f(x)=f_1(\b0)\cdot f_2(x)=f_1(\b0)\neq f_1(y)=f(y)$. Thus by the Stone-Weierstrass theorem, $\sC$ is dense in $C_c(E)$ relative to the uniform norm. On the other hand, to claim $\sC$ is dense in $\sF$ relative to the $\sE_1$-norm, we   fix $f\in \sF$ and a small constant $\varepsilon>0$. For $i=1,2$, take $g_i\in \sC^i$ with $g_i(\b0)=1$. Let $C_1:=\|g_1\|_{\sE^1_1}$ and $C_2:=\|g_2\|_{\sE^2_1}$. By \cite[Theorem~2.1.4]{FOT}, there exist  $h_{1,\varepsilon} \in \sC^1$ and $h_{2,\varepsilon}\in \sC^2$ such that 
\[
\begin{aligned}
&\|h_{1,\varepsilon}-f|_{E_1}\|_{\sE^1_1}<\varepsilon/4, \quad |h_{1,\varepsilon}(\b0)-f(\b0)|<\frac{\varepsilon}{4C_1}, \\
&\|h_{2,\varepsilon}-f|_{E_2}\|_{\sE^2_1}<\varepsilon/4, \quad |h_{2,\varepsilon}(\b0)-f(\b0)|<\frac{\varepsilon}{4C_2};  
\end{aligned}
\]
Define a function $f_\varepsilon$ on $E$ by
\[
\begin{aligned}
f_\varepsilon|_{E_1}:=h_{1,\varepsilon}+ \left(f(\b0)-h_{1,\varepsilon}(\b0) \right)\cdot g_1, \quad f_\varepsilon|_{E_2}:=h_{2,\varepsilon}+ \left(f(\b0)-h_{2,\varepsilon}(\b0) \right)\cdot g_2.
\end{aligned}
\]
Then $f_\varepsilon\in \sC$ and 
\begin{eqnarray*}
&&\|f_\varepsilon-f\|_{\sE_1}
\\
&\leq &\|f_\varepsilon|_{E_1}-f|_{E_1}\|_{\sE^1_1} + \|f_\varepsilon|_{E_2}-f|_{E_2}\|_{\sE^2_1} 
\\
&\leq & \|h_{1,\varepsilon}-f|_{E_1}\|_{\sE^1_1}+|h_{1,\varepsilon}(\b0)-f(\b0)|\cdot \|g_1\|_{\sE^1_1} +\|h_{2,\varepsilon}-f|_{E_2}\|_{\sE^2_1}+|h_{2,\varepsilon}(\b0)-f(\b0)|\cdot \|g_2\|_{\sE^2_1}  
\\
&<& \varepsilon. 
\end{eqnarray*}
This tells us that $\sC$ is dense in $\sF$ relative to the $\sE_1$-norm, which implies the regularity of $(\sE, \sF)$.  
\qed
\end{proof}

\begin{definition}[Distorted Brownian motion with varying dimension]\label{definition-dBMVD}  Let $\gamma>0$ be fixed.   The diffusion process   associated with $(\sF, \sE)$  defined in \eqref{dirichletform-dbm-1} is called  a distorted Brownian motion with varying dimension and   is denoted by $M$.
\end{definition}

\section{Basic properties of $M$ and its associated signed radial process}\label{Basics-of-M}

In this section, we quickly remark on some of the  basic properties of $M$ that are mostly reflected through its Dirichlet form expression. We give the rigorous statement that $\b0$ is of positive capacity with respect to $M$, therefore can be hit with probability one starting from quasi-everywhere. From there  the existence of the transition density function is  established. 
In the second subsection  we  give the SDE characterization for the radial process of $M$ which is needed in Section \ref{Proof-HK}. As a corollary to this SDE characterization, we present the isotropic property of $M$.

\subsection{Basic properties of $M$}\label{SEC3-1}

For any open set  $D\subset E$, we denote by   $\text{Cap}_1(D)$ the  $1$-capacity of $D$ with respect to $(\sE,\sF)$, i.e., 
\begin{equation*}
\text{Cap}_1(D):=\inf \{\sE_1(u, u):\, u\in \sF, \, u\ge 1 \quad m_\gamma\text{-a.e. on }D\}.
\end{equation*}
For an arbitrary subset $A\subset E$, 
\begin{align*}
\text{Cap}_1(A):&=\inf\{\text{Cap}_1(D):\, D\supset A,\, D \text{ open}\}
\\
&=\inf \{\sE_1(u, u):\, u\in \sF,\, \exists D \text{ open in }E \text{ s.t. } u\ge 1 \quad m_\gamma\text{-a.e. on }D\}.
\end{align*}
\begin{proposition}\label{PRO2}
For any $u\in E_2$, ${\rm Cap}_1(\{u\})>0$. However, for any $x\in E_1\setminus \{\b0\}$, ${\rm Cap}_1(\{x\})=0$. Particularly, ${\rm Cap}_1(\{\b0\})>0$. Furthermore, $\IP_x(\sigma_{\{\b0\}}<\infty)=1$ for all  $x\in E$. 
\end{proposition}
\begin{proof}
The conjuction of \cite[Theorem 1.3.14 (i)]{CF}  and \cite[Theorem 3.3.8 (iii)]{CF} states that: For $i=1,2$, any set $N\subset E_i\backslash \{\b0\}$ is ${\rm Cap}_1$-polar with respect to $(\sE, \sF)$ if and only if it is ${\rm Cap}_1$-polar with respect to $(\sE^{E_i\backslash \{\b0\}}, \sF^{E_i\backslash \{\b0\}})$. We denote the Dirichlet forms of $(\sE^i, \sF^i)$, $i=1,2$, restricted on $E_i\backslash \{\b0\}$ by  $(\sE^{i, E_i\backslash \{\b0\}}, \sF^{i, E_i\backslash \{\b0\}})$.  Thus a subset of $E_i\backslash \{\b0\}$  being ${\rm Cap}_1$-polar with respect to  $(\sE^{E_i\backslash \{\b0\}}, \sF^{E_i\backslash \{\b0\}})$ is equivalent to it  being ${\rm Cap}_1$-polar with respect to  $(\sE^{i, E_i\backslash \{\b0\}}, \sF^{i, E_i\backslash \{\b0\}})$.

 It is clear that for any $u\in E_2$,   ${\rm Cap}_1(\{u\})>0$ with respect to $(\sE^2, \sF^2)$,   and that ${\rm Cap}_1(\{x\})=0$ with respect to $(\sE^1, \sF^1)$   for any $x\in E_1\backslash \{\b0\}$. Hence, by repeatedly applying the statement in the  first paragraph, we know that ${\rm Cap}_1(\{u\})>0$  with respect to $(\sE, \sF)$  for any $u\in E_2\backslash \{\b0\}$, and ${\rm Cap}_1(\{x\})=0$ with respect to $(\sE, \sF)$ for any $x\in E_1\backslash \{\b0\}$. 
 
To show  that ${\rm Cap}_1(\{\b0\})>0$ with respect to $(\sE, \sF)$: For any open set  $D\subset E$ containing $\b0$ and  any $u\in \sF$ with $u\ge 1$ $m_\gamma$-a.e. on $D$, $u|_{E_1}\in \sF^1$, $u|_{E_2}\in \sF^2$. Also in view of the definition of open sets of $E$ in Section \ref{Sec-intro}, $D|_{E_1}$ is open in $E_1$ and $D|_{E_2}$ is open in $E_2$.  We know for both $i=1,2$,  ${\rm Cap}_1(\{\b0\})>0$ with respect to  $(\sE^i, \sF^i)$. Since $\sE_1(u, u)=\sE_1^1(u|_{E_1}, u|_{E_1})+\sE_1^2(u|_{E_2}, u|_{E_2})$. it must hold that ${\rm Cap}_1(\{\b0\})>0$ with respect to $(\sE, \sF)$.

 To claim the last assertion, we first note  the restriction of $M$ on $E_1\backslash \{\b0\}$ is homeomorphic to $X^{(3)}$ restriced on $\IR^3\backslash \{0\}$. Thus   \cite[Corollary 3.11]{FL17} implies  
 \begin{equation*}
 \IP_x\left(\sigma_{\{\b0\}}<\infty\right)=1, \quad x\in E_1\backslash \{\b0\}.
 \end{equation*}
Similarly, the restriction of $M$ on $E_2\backslash \{\b0\}$ is homeomorphic to a $1$-dimensional Brownian motion on $\IR_+$ with a constant drift, it is   also clear that 
\begin{equation*}
 \IP_x\left(\sigma_{\{\b0\}}<\infty\right)=1, \quad x\in E_2 \backslash \{\b0\}.
 \end{equation*}
 Finally, to show that $\IP_\b0\left(\sigma_{\{\b0\}}<\infty\right)=1$, we prove by contradiction that actually $\IP_\b0\left(\sigma_{\{\b0\}}=0\right)=1$. If not, then by 0-1 law it would hold $\IP_\b0\left(\sigma_{\{\b0\}}=0\right)=0$. By \cite[Definition A.2.6]{CF}, $\{\b0\}$ would be a thin therefore semipolar set with respect to $M$. By \cite[Theorem 3.1.10]{CF},. $\{\b0\}$ would be $\sE$-polar, which then by \cite[Theorem 3.3.8(iii)]{CF} would imply that $\{\b0\}$ is ${\rm Cap}_1$-polar. This contradicts with the proved first assertion. 
\end{proof}
\qed

With the proposition above, we  establish   the existence of transition density as follows. 

\begin{proposition}
Let $(P_t(x, \cdot))_{t\geq 0}$ be the transition semigroup of $M$. For any $x\in E, t> 0$, $P_t(x,\cdot)\ll m_\gamma$. Thus there exists a density function $\{p(t,x,y): t>0, x,y\in E\}$ such that $P_t(x,dy)=p(t, x,y)m_\gamma(dy)$. 
\end{proposition}
\begin{proof}
By \cite[Theorem 4.2.4]{FOT}, it suffices to show that any $m_\gamma$-ploar set is polar. with respect to $M$  Let $B$  be  an arbitrary $m_\gamma$-polar set. ByProposition \ref{PRO2}, $\text{Cap}_1(\{x\})>0$ for any $x\in E_2$. Therefore $B\cap E_2 =\emptyset$, i.e., $B\subset (E_1\backslash \{\b0\})$. It then follows from the continuity of $M$ and the fact $\b0$ is regular for itself  that 
\begin{equation}\label{e:3.1}
\IE_x\left[e^{-\sigma_B}, \, \sigma_B<\sigma_{\{\b0\}}\right]=0, \quad \forall x\in E_2.
\end{equation}
We denote by $M^{E_1\backslash \{\b0\}}$ the part process of $M$ on $E_1\backslash \{\b0\}$, which is also equivalent to  (up to the isomorphism $\iota_1$) $X^{(3)}$ killed upon hitting $0 \in \IR^3$. By \cite[Theorem 3.3.8]{CF}, the assumption $B$ is $m_\gamma$-polar with respect to $M$  together with the fact we just showed that $B\subset (E_1\backslash \{\b0\})$ implies that $B$ is $m_\gamma$-polar with respect to $M^{E_1\backslash \{\b0\}}$.   It was shown in \cite[Remark 2.2]{FL17} that $M^{E_1\backslash \{\b0\}}$ has a density with respect to $m_\gamma|_{E_1}$, thus by \cite[Theorem A.2.17, Corollary 3.1.14]{CF} , $B$ is polar with respect to $M^{E_1\backslash \{\b0\}}$.  We now have
\begin{equation}\label{e:3.2}
\IE_x\left[e^{-\sigma_B}, \, \sigma_B<\sigma_{\{\b0\}}\right]=0, \quad \forall x\in E_1\backslash \{\b0\}.
\end{equation}
\eqref{e:3.1} and \eqref{e:3.2} together shows that $\IE_x\left[e^{-\sigma_B}, \, \sigma_B<\sigma_{\{\b0\}}\right]=0$  for all $x\in E$. Now for any $x\in E$, by the strong Markov property of $M$, it holds on $E$ that
\begin{align}
\IE_x\left[e^{-\sigma_B}\right]& =\IE_x\left[e^{-\sigma_B}, \, \sigma_B<\sigma_{\{\b0\}}\right]+\IE_x\left[e^{-\sigma_B}, \, \sigma_B\ge \sigma_{\{ \b0\}} \right] \nonumber
\\
&=\IE_x\left[e^{-\sigma_B}, \, \sigma_B\ge \sigma_{\{\b0\}}\right]  \nonumber
\\
&=\IE_x\left[e^{-\sigma_{\{\b0\}}}\cdot e^{-\sigma_B\circ \theta_{\sigma_{\{\b0\}}}}\cdot \mathbf{1}_{\left\{\sigma_B\ge \sigma_{\{\b0\}}\right\}}\right]  \nonumber
\\
&=\IE_x\left[\IE_x\left[e^{-\sigma_{\{\b0\}}}\cdot e^{-\sigma_B\circ \theta_{\sigma_{\{\b0\}}}}\cdot \mathbf{1}_{\left\{\sigma_B\ge \sigma_{\{\b0\}}\right\}}\bigg|\mathcal{F}_{\sigma_{\{\b0\}}}\right]\right]\nonumber
\\
&\le \IE_x\left[ e^{-\sigma_{\{\b0\}}}  \IE_x\left[ e^{-\sigma_B\circ \theta_{\sigma_{\{\b0\}}}}\big|\mathcal{F}_{\sigma_{\{\b0\}}} \right]  \right] \nonumber
\\
&\le \IE_x\left[ \IE_{M_{\sigma_{\{\b0\}}}}\left[ e^{-\sigma_B} \right]  \right] =\IE_x\left[\IE_\b0\left[e^{-\sigma_B}\right]\right]=0, \quad \text{q.e. }  x. \label{e:3.3}
\end{align}
The last equality in \eqref{e:3.3} is due to the fact that the map $x\mapsto \IE_x\left[e^{-\sigma_B}\right]$ is finely continous and vanishes $m_\gamma$-a.e., therefore it vanishes q.e. (see, for example, \cite[Lemma 4.1.5]{FOT}, which implies that $\IE_\b0\left[e^{-\sigma_B}\right]=0$. Now that we have shown that for q.e. $x$, $\IE_x[e^{-\sigma_B}]=0$ for an arbitrary $m_\gamma$-polar set $B$, it has been proved that any $m_\gamma$-polar set is polar. This completes the proof. 
\end{proof}\qed

\subsection{Signed radial process of $M$ and its isotropic property}\label{S:3.2}

To introduce the signed radial process of $M$, we define 
\begin{equation}\label{definition-radial-process}
	u(x):=\left\lbrace
		\begin{aligned}
		|x|,\quad &  x \in E_1,\\
		-|x|,\quad & x\in E_2
		\end{aligned} \right.
\end{equation}
and let $Y_t:=u(M_t)$ for $t\geq 0$. 
\begin{proposition}
$Y=(Y_t)_{t\geq 0}$ is a symmetric diffusion process on $\bR$ with respect to $\wt{\fm}$ and can be characterized by the following SDE:
\begin{equation}\label{Y-SDE-alpha=gamma}
	Y_t-Y_0=B_t+\gamma \int_0^t 1_{(-\infty, 0)}(Y_s)ds-\gamma\int_0^t 1_{(0,\infty)}(Y_s)ds,\quad t\geq 0, 
\end{equation}
where $(B_t)_{t\ge 0}$ is a  $1$-dimensional standard Brownian motion.
\end{proposition}
\begin{proof}
Set $\wt{\mathfrak{m}}:=m_\gamma \circ u^{-1}$, which is a fully supported Radon measure on $\bR$. By simple computation one can easily obtain 
\begin{equation}\label{def-m-tilde}
	\wt{\fm}(dx)=2\gamma \re^{-2\gamma |x|}dx.
\end{equation}
We first  show that $Y$ is a symmetric Markov process using \cite[Theorem 13.5]{Sharpe} by verifying conditions (13.1)-(13.3) in \cite{Sharpe}, among which (13,1) and (13.3) are obvious. To show (13.2), we first denote by  $\cE^u(\bR)$  the family of all universally measurable functions on $\bR$. Let $(P_t)_{t\ge 0}$ be the semigroup of $M$. We need  to show that for any bounded $f\in \mathcal{E}^u(\bR)$, there exists $g\in \cE^u(\bR)$ such that
\begin{equation}\label{EQ4PTF}
	P_t(f\circ u)=g\circ u.
\end{equation}	
By the rotational invariance of $M$ on $E_1$, it is not hard to see that  for any $x, y \in E$ with $u(x)=u(y)$, 
\[
	\int_E f(u(\cdot))\mathbb{P}_x(M_t\in \cdot)=\int_E f(u(\cdot)) \mathbb{P}_y(M_t\in \cdot). 
\]
This implies $P_t(f\circ u)(x)=P_t(f\circ u)(y)$. For $r\in \IR$ and $x\in E$, such that $r=u(x)$, set  $g(r):=P_t(f\circ u)(x)$ which a well-defined function on $\bR$ since $u$ is surjective.  Since $u$ is continuous, $f\circ u\in \cE^u(E)$ and thus $P_t(f\circ u)\in \cE^u(E)$.  Finally to verify the universal measurability of $g$, for any set $A\in \cB(\bR)$, we let $B_+:=g^{-1}(A)\cap (0,\infty)$ and $B_-:=g^{-1}(A)\cap (-\infty, 0]$. In the following we claim actually both $B_+, B_- \in \cE^u(\bR)$. Denote by  $S^2:=\{x\in \bR^3: |x|=1\}$.  Observe that 
\begin{equation*}
\iota_1(B_+\times S^2)= \bigg((P_t(f\circ u))^{-1}(A)\bigg)\cap (E_1\setminus \{\b0\})\in \cE^u(E)
\end{equation*}
 and 
 \begin{equation*}
 \iota_2(-B_-)=\bigg((P_t(f\circ u))^{-1}(A)\bigg)\cap E_2\in \cE^u(E). 
 \end{equation*}
In view of the continuity of $\iota_1, \iota_2$, we have $B_+, B_-\in \cE^u(\bR)$. Thus $g^{-1}(A)=B_+\cup B_-\in \cE^u(\bR)$. 
Now  \cite[Theorem 13.5]{Sharpe} yields that $Y$ is a Markov process with transition semigroup 
\[
	P^Y_tf:=g, \quad \text{for }f\in \cE^u(\IR).
\]
To verify  the symmetry of $Y$:  For any two functions $f_1,f_2\in \cE^u(\IR)$, it holds
\begin{equation*}
	(P^Y_t f_1,f_2)_{\wt{\fm}}=((P^Y_tf_1)\circ u, f_2\circ u)_{m_\gamma}=(P_t(f_1\circ u), f_2\circ u)_{m_\gamma}=(f_1\circ u, P_t(f_2\circ u))_{m_\gamma}=(f_1, P^Y_t f_2)_{\wt{\fm}}.
\end{equation*}
It   follows that $Y$ is associated with the Dirichlet form on $L^2(\bR,\wt{\fm})$:
	\begin{equation*}
	\left\{
		\begin{aligned}
			&\sF^Y=\{f: f\circ u\in \sF\}, \\
		&\sE^Y(f,f)=\sE(f\circ u, f\circ u),\quad f\in \sF^Y. 
		\end{aligned}
		\right.
	\end{equation*}
 By a simple computation,  the above can be rewritten as
\begin{equation}\label{EQ4FYF}
\left\{
\begin{aligned}
	&\sF^Y=\{f\in L^2(\bR, \wt{\fm}): f'\in L^2(\bR, \wt{\fm})\}, \\
	&\sE^Y(f,g)=\frac{1}{2}\int_\bR f'(x)g'(x)\wt{\fm}(dx),\quad f,g\in \sF^Y. 
\end{aligned}
\right.
\end{equation}
Next we take $f(x):=x\in \sF^Y_\mathrm{loc}$ and consider the Fukushima's decomposition (whose definition can be found in, e.g., \cite[Theorem 4.2.6]{CF}):
\[
	f(Y_t)-f(Y_0)=M^f_t+N^f_t. 
\]
The martingale part $M^f$ is determined by its energy measure $\mu_{\<f\>}$ and for any $g\in C_0^\infty(\bR)$,
\[
	\int gd\mu_{\<f\>}= 2\sE^Y(fg,f)-\sE^Y(f^2,g)=\int gd\wt{\fm}. 
\]
It follows that $\mu_{\<f\>}=\wt{\fm}$ and hence $M^f$ has the same distribution as one-dimensional Brownian motion. For the zero-energy part $N^u$, we note
\[
	-\sE^Y(f,g)=-\frac{1}{2}\int_\bR g'(x)\wt{\fm}(dx)=\gamma \int_{-\infty}^0 g(x)\wt{\fm}(dx)- \gamma \int_0^\infty g(x)\wt{\fm}(dx). 
\]
Thus $N^u$ is of bounded variation, and 
\[
	\mu_{N^u}=\gamma \cdot \wt{\fm}|_{(-\infty, 0)}-\gamma\cdot \wt{\fm}|_{(0,\infty)}. 
\]
Eventually, it follows from \cite[Theorem 5.5.5]{FOT} that
\[
	Y_t-Y_0=B_t+\gamma \int_0^t 1_{(-\infty, 0)}(Y_s)ds-\gamma\int_0^t 1_{(0,\infty)}(Y_s)ds, \quad t\geq 0, 
\]	
where $(B_t)_{t\ge 0}$ is a standard Brownian motion.
\end{proof}
\qed

As a corollary to Proposition \ref{Y-SDE-alpha=gamma} , we mention the isotropic property and rotational invariance of $M$ as follows.
\begin{corollary}
\begin{equation*}
\left|M^{E\backslash \{\b0\}}\right|\overset{d}{=} \left|M^{E_1\backslash 
\{\b0\}}\right|\overset{d}{=} M^{E_2\backslash \{\b0\}}.
\end{equation*} 
\end{corollary}
\begin{proof}
We first notice that $\left|M^{E\backslash \{\b0\}}\right|=|Y^{\IR\backslash \{0\}}|$,  $\left|M^{E_1\backslash \{\b0\}}\right| = Y^{(0, +\infty)}$, and $\left|M^{E_2\backslash \{\b0\}}\right|=-Y^{(-\infty, 0)}$. 
 Proposition \ref{Y-SDE-alpha=gamma} suggests that $-Y\overset{d}{=}Y$, thus both identities follow.  
\end{proof}
\qed

\section{Heat kernel of $M$: Proof to Theorem \ref{main-thm}}\label{Proof-HK}

Throughout the rest of this paper, we define the ``unsigned radial process" of $M$ as
\begin{equation}\label{def-Yhat}
\widehat{Y}:=|M|=|Y|.
\end{equation}
For $M$, we use $\wh{p}(t,x,y)$ to denote the transition density with respect to the measure on $E$ induced by $3$- or $1$-dimensional Lebesgue measure, and we  let $p(t,x,y)$ denote the transition density with respect to $m_\gamma$. Thus  
\begin{equation}\label{relationship_p_phat}
p(t,x,y)=\wh{p}(t,x,y)\frac{1}{h_\gamma (y)^2},
\end{equation}
where
\begin{eqnarray*}
	h_{\gamma}:=\left\lbrace
	\begin{aligned}
		&\psi_\gamma, \quad \text{ on }E_1, \\
		&\phi_\gamma,\quad \text{ on }E_2,
	\end{aligned} \right.
\end{eqnarray*}
where $\phi_\gamma$ is defined in \eqref{def-phi_gamma}.  We denote by $p^Y(t,x,y)$ and $p^{\wh{Y}}(t,x,y)$ the densities of $Y$ and $\wh{Y}$ respectively, both  with respect to $\wt{\mathfrak{m}}$ characterized in \eqref{def-m-tilde}. We denote by $\wh{p}^Y(t,x,y)$ and $\wh{p}^{\wh{Y}}(t,x,y)$  the densities of $Y$ and $\wh{Y}$ with respect to the $1$-dimensional Lebesgue measure.

The  first key ingredient of the proof is that we establish the explicit transition density for $3$-dimensional  dBM killed upon hitting $\b0$. The second key ingredient is to find the explicit density function for part dBM with varying dimension restricted on $E_2$. The global density function for $M$ can be obtained by combining these two key ingredients, as well as the exact density function for $1$-dimensional Brownian motion with constant drift (pushing towards $0$) reflected at $0$, which was established in \cite{Linetsky}.

Recall that we denote the $3$-dimensional distorted Brownian motion by $X^{(3)}$. We let $q(t,x,y)$ denote the transition density function of the part process of $M$ on $(E_1\backslash \{\b0\})$ killed upon hitting $\b0$ with respect to $m_\gamma^{1}$, i.e., for any non-negative function $f\geq 0$ on $E_1$, 
\begin{equation}\label{e:1.7}
 \int_{E_1\backslash \{\b0\}}  q (t, x, y)  f(y) m_\gamma^{1}(dy) = \IE_x \left[ f(M_t); t< \sigma_{\{\b0\}} \right], \quad x\in E_1\backslash \{\b0\}
\end{equation}
\begin{proposition}\label{HK-killed-dBM}
\begin{equation}\label{EQQTX}
q(t,x,y)=\frac{1}{(2\pi t)^{3/2}}\;e^{-\gamma^2t/2-|x-y|^2/(2t)}\frac{1}{\psi_\gamma (x)\psi_\gamma (y)}, \quad x,y\in E_1\backslash \{\b0\}, \, t>0. 
\end{equation}
\end{proposition}
\begin{proof}
We first observe that  $q(t,\iota^{-1}(x), \iota^{-1}(y))$ coincides with  the  transition density of the part process of  $X^{(3)}$ on $\IR^3\backslash \{0\}$.   
Now denote by $X^{(3), 0}$  the part process of $X^{(3)}$  killed upon hitting $0$. It is associated with the Dirichlet form $(\mathcal{E}^{(3), 0},\mathcal{F}^{(3),0})$ on $L^2(\IR^3, \fm^{(3)})$ where 
\begin{equation}
\left\{
\begin{aligned}
&\mathcal{F}^{(3),0}=\{f\in \mathcal{F}^{(3)}: \tilde{f}(0)=0\}, \\
&\mathcal{E}^{(3),0}(f,g)=\mathcal{E}^{(3)}(f,g),\quad f,g\in \mathcal{F}^{(3),0}. 
\end{aligned}
\right.
\end{equation}
Note that $\mathcal{C}_0:=C_c^\infty(\IR^3\setminus \{0\})$ is a special standard core of $(\mathcal{E}^{(3),0},\mathcal{F}^{(3),0})$.   Set
\begin{equation}\label{EQGUL}
\left\{
\begin{aligned}
	&\mathcal{G}^{(3)}:=\{u\in L^2(\IR^3, dx): u/\psi_\gamma \in \mathcal{F}^{(3),0}\}, \\
&\mathcal{A}^{(3)}(u,v):=\mathcal{E}^{(3),0}(u/\psi_\gamma,v/\psi_\gamma),\quad u,v\in \mathcal{G}. 
\end{aligned}
\right.
\end{equation}
It is easy to verify that $(\mathcal{A}^{(3)},\mathcal{G}^{(3)})$ is a closed form on $L^2(\IR^3, dx)$ and $\mathcal{C}_0\cdot \psi_\gamma:=\{f\cdot \psi_\gamma: f\in \mathcal{C}_0\}$ is $\mathcal{A}^{(3)}_1$-dense in $\mathcal{G}^{(3)}$.  Since $\psi_\gamma$ is smooth on $E_1\setminus \{0\}$, it follows that $\mathcal{C}_0\cdot \psi_\gamma=\mathcal{C}_0$. Hence $\mathcal{C}_0$ is $\mathcal{A}^{(3)}_1$-dense in $\mathcal{G}^{(3)}$. Take $u,v\in \mathcal{C}_0$. Mimicking the proof of \cite[Theorem~2.1]{FL17}, one can obtain
\[
	\mathcal{A}^{(3)}(u,v)=\mathcal{E}^{(3),0}(u/\psi_\gamma, v/\psi_\gamma)=\frac{1}{2}\int_{E_1}\nabla u(x) \cdot \nabla v(x)dx+\frac{\gamma^2}{2}\int_{E_1}u(x)v(x)dx. 
\]
As a result, $\mathcal{G}^{(3)}=H^1(\IR^3)$ and $(\mathcal{A}^{(3)},\mathcal{G}^{(3)})$ is a regular Dirichlet form on $L^2(\IR^3, dx)$ associated with the Brownian motion killed at the rate $\gamma^2/2$. From \eqref{EQGUL}, we can eventually conclude \eqref{EQQTX}, which completes the proof. 
\end{proof}
\qed

Recall that  it has been defined in \eqref{def-Yhat} that $\wh{Y}=|Y|$. The following proposition says that $\wh{Y}$ can be viewed as a  reflected Brownian motion with a constant drift.
\begin{proposition}\label{SDE-mod-M}
\begin{equation*}
d\wh{Y}_t =dB_t-\gamma dt+dL_t^0,\quad t\ge 0,
\end{equation*}
where $L^0$ is the symmetric semimartingale local time with respect to $\wh{Y}$ defined as follows:
\begin{equation*}
\widehat{L}_t^0(\wh{Y}):=\lim_{\delta \downarrow 0}\frac{1}{2 \delta }\int_0^t 1_{(-\delta, \delta )}(\wh{Y}_s)d\langle \wh{Y}\rangle_s
= \lim_{\delta \downarrow 0}\frac{1}{2 \delta}\int_0^t 1_{(-\delta, \delta )}(\wh{Y}_s) d s. 
\end{equation*}

\end{proposition}
\begin{proof}
This is an immediate consequence of applying Tanaka's formula to \eqref{Y-SDE-alpha=gamma}.
\end{proof}
\qed

The following  transition density (with respect to Lebesgue measure) of reflected Brownian motion with constant drift was established by Linetsky in \cite[Section 4.2]{Linetsky}: 
\begin{align*}
\wh{p}^{\wh{Y}}(t, x, y)&=2\gamma e^{-2\gamma y} +\frac{2}{\pi}e^{\gamma (x-y)-\gamma^2 t/2} \nonumber
\\
&\times \int_{0}^\infty \frac{e^{-s^2t/2}}{s^2+\gamma^2}\left[s \cos(sx)-\gamma \sin(sx)\right]\left[s\cos(sy)-\gamma \sin (sy)\right]ds, \, x,y\in (0, +\infty).
\end{align*}
By a simple change of measure, we get
\begin{eqnarray}
&& p^{\wh{Y}}(t, x, y)=\wh{p}^{\wh{Y}}(t, x, y)\frac{1}{\phi_\gamma(y)^2}\nonumber
\\
&=& 1 +\frac{1}{\pi \gamma}e^{\gamma (x+y)-\gamma^2 t/2}\int_{0}^\infty \frac{e^{-s^2t/2}}{s^2+\gamma^2}\left[s \cos(sx)-\gamma \sin(sx)\right]\left[s\cos(sy)-\gamma \sin (sy)\right]ds, \, x,y\in (0, +\infty). \nonumber
\\\label{density-RBM-drift}
\end{eqnarray}

Let  $p_{E_2\backslash \{\b0\}}(t, x,y)$ denote the transition density of the part process of  $M$ restricted on $E_2\backslash \{\b0\}$ with respect to $m_\gamma^2$.  We  first record the following lemma regarding $p_{E_2\backslash \{\b0\}}(t, x,y)$. 

\begin{lemma}
\begin{equation}\label{density-p_E2}
p_{E_2\backslash \{\b0\}}(t,x,y)=\frac{1}{\gamma\sqrt{8\pi t}}e^{-\gamma^2 t/2+\gamma(|x|+|y|)}\left(e^{-(|x|-|y|)^2/(2t)}-e^{-(|x|+|y|)^2/(2t)}\right), x,y\in E_2\backslash \{\b0\},  t>0.
\end{equation}
\end{lemma}
\begin{proof}
The idea of this proof is very similar to that of Proposition \ref{HK-killed-dBM}.  Let $X^{(+), 0}$ be the part process of $X^{(+)}$ on $\IR_+$ killed upon hitting $0$.  The density of $X^{(+), 0}$ coincides with  $p_{E_2\backslash \{\b0\}}(t,\iota^{-1}(x),\iota^{-1}(y))$. Recall $(\mathcal{E}^{(+)}, \mathcal{F}^{(+)})$ has  been  defined in  \eqref{def-EE+}.  $X^{(+), 0}$ is associated with the Dirichlet form $(\mathcal{E}^{(+),0},\mathcal{F}^{(+),0})$ on $L^2(\IR_+\cup \{0\}, \fm^{(+)})$ where 
\begin{equation*}
\left\{
\begin{aligned}
&\mathcal{F}^{(+),0}=\{f\in \mathcal{F}^{(+)}: f(0)=0\}, \\
&\mathcal{E}^{(+),0}(f,g)=\mathcal{E}^{(+)}(f,g),\quad f,g\in \mathcal{F}^{(+),0}.
\end{aligned}
\right.
\end{equation*}
Note that $C_c^\infty(\IR_+)$ is a special standard core of $(\mathcal{E}^{(+),0},\mathcal{F}^{(+),0})$.  Recall that on $\IR_+\cup \{0\}$  it is defined in \eqref{def-phi_gamma}   that $\phi_\gamma (u):=\sqrt{2\gamma} e^{-\gamma u}$.        Set
\begin{equation}
\left\{
\begin{aligned}
	&\mathcal{G}^{(+)}:=\{u\in L^2(\IR_+\cup \{0\}, dx): u/\phi_\gamma \in \mathcal{F}^{(+),0}\}, \\
&\mathcal{A}^{(+)}(u,v):=\mathscr{E}^{(+),0}(u/\phi_\gamma,v/\phi_\gamma),\quad u,v\in \mathcal{G}^{(+)}. 
\end{aligned}
\right.
\end{equation}
It is easy to see that   $(\mathcal{E}^{(+),0},\mathcal{F}^{(+),0})$ is an $h$-transform of $(\mathcal{A}^{(+)},\mathcal{G}^{(+)})$, where $h=\phi_\gamma$.  In the following we  show:  
\begin{equation}\label{G+}
\mathcal{G}^{(+)}=\left\{f\in L^2(\mathbb{R}_+\cup \{0\}, dx):  f(0)=0, f \text{ absolutely continuous on }[0, +\infty), f'\in L^2(\mathbb{R}_+,dx)\right\}
\end{equation}
 and  $(\mathcal{A}^{(+)},\mathcal{G}^{(+)})$ is a regular Dirichlet form on $L^2(\IR_+\cup \{0\}, dx)$ associated with the $1$-dimensional Brownian motion restricted  on $\IR_+$ killed at a rate $\gamma^2/2$. The approach is similar to Proposition \ref{HK-killed-dBM}. Below we spell out the details. 
It is easy to verify that  $(\mathcal{A}^{(+)},\mathcal{G}^{(+)})$ is a closed form on $L^2(\IR_+\cup \{0\}, dx)$ and $C_c^\infty(\IR_+)\cdot \phi_\gamma :=\{f\cdot \phi_\gamma: f\in C_c^\infty(\IR_+)\}$ is $\mathcal{A}^{(+)}_1$-dense in $\mathcal{G}^{(+)}$.  Since $\phi_\gamma$ is smooth on $(0, +\infty)$,  $C_c^\infty(\IR_+)\cdot \phi_\gamma=C_c^\infty(\IR_+)$. Hence $C_c^\infty(\IR_+)\}$ is $\mathcal{A}^{(+)}_1$-dense in $\mathcal{G}^{(+)}$. Taking $u,v\in C_c^\infty(\IR_+)\}$, we have
\begin{align*}
\mathcal{E}^{(+)}(u, v)&=\frac{1}{2}\int_{\IR_+}u'(x)v'(x)2\gamma e^{-2\gamma |x|}dx
\\
&=-\frac{1}{2}\int_{\IR_+}u(x)\frac{d}{dx}\left(v'(x)2\gamma e^{-2\gamma |x|}\right)dx
\\
&=-\frac{1}{2}\int_{\IR_+}u(x)v''(x)2\gamma e^{-2\gamma |x|}dx-\frac{1}{2}\int_{\IR_+}u(x)v'(x)(-4\gamma^2)e^{-2\gamma |x|}dx
\\
&=\left(-\frac{1}{2}v''(x)+\gamma v'(x), u(x)\right)_{2\gamma e^{-2\gamma |x|}dx}.
\end{align*}
Therefore,
\begin{align*}
\mathcal{A}^{(+)}(u,v)&=\mathcal{E}^{(+),0}(u/\phi_\gamma, v/\phi_\gamma)
\\
&=\left(-\frac{1}{2}\left(\frac{v(x)}{\sqrt{2\gamma}}e^{\gamma |x|}\right)''+\gamma \left(\frac{v(x)}{\sqrt{2\gamma}}e^{\gamma |x|}\right)', u(x)\phi(x)^{-1}\right)_{2\gamma e^{-2\gamma |x|}dx}
\\
&=\Bigg(-\frac{1}{2}\bigg(\frac{v''(x)}{\sqrt{2\gamma}}e^{\gamma |x|}+\frac{v'(x)\gamma}{\sqrt{2\gamma}}e^{\gamma |x|}+v'(x)\frac{\gamma}{\sqrt{2\gamma}}e^{\gamma |x|}+v(x)\frac{\gamma^2}{\sqrt{2\gamma}}e^{\gamma |x|}\bigg)
\\
&\quad \quad+\gamma v'(x)\frac{1}{\sqrt{2\gamma}}e^{\gamma |x|}+\frac{\gamma^2}{\sqrt{2\gamma}}v(x)e^{\gamma |x|},\, u(x)\phi(x)^{-1}\Bigg)_{2\gamma e^{-2\gamma |x|}dx}
\\
&=\left(-\frac{1}{2}\frac{v''(x)}{\sqrt{2\gamma}}e^{\gamma |x|}+\frac{1}{2}\frac{\gamma^2}{\sqrt{2\gamma}}v(x)e^{\gamma |x|}, \, \frac{u(x)}{\phi(x)}\right)_{2\gamma e^{-2\gamma |x|}dx} 
\\
&=\left(-\frac{1}{2}\frac{v''(x)}{\phi(x)}+\frac{1}{2}\gamma^2 \frac{v(x)}{\phi(x)}, \, \frac{u(x)}{\phi(x)}\right)_{2\gamma e^{-2\gamma |x|}dx}
\\
&=-\frac{1}{2}\int_{\IR_+}v''(x)u(x)dx+\frac{\gamma^2}{2}\int_{\IR_+}v(x)u(x)dx
\\
&=\frac{1}{2}\int_{\IR_+}u'(x)v'(x)dx+\frac{\gamma^2}{2}\int_{\IR_+}v(x)u(x)dx.
\end{align*}
Consequently, \eqref{G+} holds and $(\mathcal{A}^{(+)},\mathcal{G}^{(+)})$ is a regular Dirichlet form on $L^2(\IR_+\cup \{0\})$ associated with the part Brownian motion on $\IR_+$   killed at the ratio $\gamma^2/2$. Recall that $(\mathcal{E}^{(+),0},\mathcal{F}^{(+),0})$ is an $h$-transform of $(\mathcal{A}^{(+)},\mathcal{G}^{(+)})$, and that the density of $X^{(+), 0}$ is the same as $p_{E_2\backslash \{\b0\}}(t,\iota^{-1}(x),\iota^{-1}(y))$. Since the transition density of  the part process of 1-dimensional Brownian motion restricted  on $(0, +\infty)$  is explicitly known,   we can eventually conclude 
\begin{equation*}
p_{E_2\backslash \{\b0\}}(t,x,y)=\frac{1}{\sqrt{2\pi t}}e^{-\gamma^2 t/2}\left(e^{-|x-y|^2/(2t)}-e^{-|x+y|^2/(2t)}\right)\frac{1}{\phi_\gamma (x)\phi_\gamma (y)}, \,x,y\in E_2\backslash \{\b0\}, \, t>0,
\end{equation*}
where $\phi_\gamma (x)=\sqrt{2\gamma}e^{-\gamma |x|}$, i.e., 
\begin{equation*}
p_{E_2\backslash \{\b0\}}(t,x,y)=\frac{1}{\gamma\sqrt{8\pi t}}e^{-\gamma^2 t/2+\gamma(|x|+|y|)}\left(e^{-|x-y|^2/(2t)}-e^{-|x+y|^2/(2t)}\right), \, x,y\in E_2\backslash \{\b0\}, \, t>0.
\end{equation*}
\end{proof}
\qed

\begin{remark}
We notice that $ E_2\backslash \{\b0\}=\{(0_3, x): x\in \mathbb{R}_+\}\cong \mathbb{R}_+$, and the dBM $M$ restricted on $E_2\backslash \{\b0\}$ has the same distribution as $Y$ restricted on $\IR_-$ by switching the sign. Also $Y$ is symmetric about zero, i.e., 
\begin{equation*}
p^Y_{\IR_-}(t,-x,-y)=p^Y_{\IR_+}(t,x,y), \quad t>0, \,x,y\in (0, +\infty).
\end{equation*}
Namely,
\begin{align}
p^Y_{\IR_+}(t,|x|,|y|)&=p^Y_{\IR_-}(t,-|x|,-|y|)=p^M_{E_2\backslash \{\b0\}}(t,x,y) \nonumber
\\
&=\frac{1}{\gamma\sqrt{8\pi t}}e^{-\gamma^2 t/2+\gamma(|x|+|y|)}\left(e^{-|x-y|^2/(2t)}-e^{-|x+y|^2/(2t)}\right),\, x,y\in E_2\backslash \{\b0\},\, t>0. \label{density-pY_R-}
\end{align}
\end{remark}

The next proposition follows from the symmetry of $Y$ with respect to $\wt{\mathfrak{m}}$.
\begin{proposition}\label{continuity-density-Y-m}
$Y$ has a transition density  $\{p^Y(t, x,y):t>0, x, y\in \IR\}$  with respect to  $\wt{\mathfrak{m}}$ which is jointly continuous on $(0, +\infty)\times \IR\times \IR$.
\end{proposition}
\begin{proof}
Since $Y$ is Brownian motion with bounded drift,  using the same argument as that for Theorem A in Zhang \cite[\S 4]{Z1}, we can show that the transition density of $Y$ with respect to Lebesgue measure exists. Denote this density by $\wh{p}^{Y}(t,x,y)$. It follows that the density of $Y$ with respect to $\wt{\mathfrak{m}}$ also exists and satisfies 
\begin{equation}\label{densities-Y}
\wh{p}^Y(t,x,y)dy = p^Y(t,x,y)\wt{\mathfrak{m}}(dy)=p^{Y}(t,x,y)2\gamma e^{-2\gamma |y|}dy, \quad \text{on }(0, +\infty)\times \IR\times \IR.
\end{equation}
Since $Y$ is a symmetric diffusion process with respect to $\wt{\mathfrak{m}}$, $p^Y(t,x,y)$ is jointly continous on $(0, +\infty)\times \IR\times \IR$.
\end{proof}
\qed

The following corollary is an immediate result following from Proposition \ref{continuity-density-Y-m} and \eqref{densities-Y}
\begin{corollary}\label{joint-cont-Y}
$Y$ has a transition density  $\wh{p}^Y(t, x,y)$  with respect to 1-dimensional Lebesgue measure which is jointly continuous on $(0, +\infty)\times \IR\times \IR$.
\end{corollary}

Before presenting the proof of finding the global transition density for $M$, we first record the following  two  simple propositions which  will be used  repeatedly.

\begin{proposition}\label{strongMarkov}
Given a strong Markov process $X$ with state space $E^{(X)}$ equipped with measure $m^{(X)}$.  Assume $X$  is continuous,  and $X$  has  transition density $\{p^X(t,x, y):\, t>0, x,y\in E^{(X)}\}$ with respect to $m^{(X)}$. Given $z\in E^{(X)}$, it holds that 
\begin{equation}
p^X(t, x,y; \;\sigma_{\{z\}}\le t)=\int_0^t p^X(t-s, z, y)\IP_x\left[\sigma_{\{z\}} \in ds \right].
\end{equation}
\end{proposition}
\begin{proof}
First by the same computation as that in \cite[pp.13]{PortStone}, it holds for any open set $A\subset E^{(X)}$ that
\begin{equation*}
\IP_x\left[X_t \in A;\, \sigma_{\{z\}}\in A\right]=\IE_x\left[\int_A p(t-\sigma_{\{z\}}, X_{\sigma\{z\}}, y)dy;\, \sigma_{\{z\}}\le t\right].
\end{equation*}
Now we define a sequence of discrete approximation to $\sigma_{\{z\}}$ as follows:
\begin{equation*}
\sigma_n:=\sum_{k=1}^\infty\frac{kt}{2^n}\cdot\mathbf{1}_{\{\frac{kt}{2^n}\le \sigma_{\{z\}}<\frac{(k+1)t}{2^n}\}}.
\end{equation*}
$\sigma_n$ increases to $\sigma_{\{z\}}$ on $\{\sigma_{\{z\}}\le t\}$. 
The facts that $X$ is continuous and that $A$ is open in $E^{(X)}$ now implies
\begin{align*}
\IP_x\left[X_t \in A;\, \sigma_{\{z\}}\in A\right]&= \IE_x\left[\int_A p(t-\sigma_{\{z\}}, z, y)dy;\, \sigma_{\{z\}}\le t\right]
\\
&=\IE_x\left[\IP_z\left[X_{t-\sigma_{\{z\}}}\in A\right]; \sigma_{\{z\}}\le t\right]
\\
&=\lim_{n\rightarrow \infty} \IE_x\left[\IP_z\left[X_{t-\sigma_n}\in A\right]; \sigma_n\le t\right]
\\
&=\lim_{n\rightarrow \infty}\sum_{k=1}^{2^n}\IE_x\left[\IP_z\left[X_{t-\frac{kt}{2^n}}\in A\right];  \sigma_{n}= \frac{kt}{2^n}\right]
\\
&=\lim_{n\rightarrow \infty} \sum_{k=1}^{2^n} \IP_z\left[X_{t-\frac{kt}{2^n}}\in A\right] \IP_x\left[  \frac{kt}{2^n}\le \sigma_{\{z\}}< \frac{(k+1)t}{2^n} \right]
\\
&= \int_0^t \IP_z\left[X_{t-s}\in A\right] \cdot \IP_x\left[\sigma_{\{z\}}\in ds\right].
\end{align*}
Since the above holds for any open set  $A
\subset E^{(X)}$,   the desired statement now follows from  \eqref{e:1.7}.
\end{proof}
\qed

\begin{proposition}
It holds for all $x, y\in E\backslash \{0\}$ and $t>0$ that
\begin{description}
\item{(i)}
\begin{equation}\label{e:4.9}
\wh{p}^{Y}(t, |x|, |y|; \sigma_{\{0\}}<t)+\wh{p}^{Y}(t, |x|, -|y|)+\wh{p}^{Y}_{\IR_+}(t, |x|, |y|)=\wh{p}^{\wh{Y}}(t, |x|, |y|).
\end{equation}
 \item{(ii)}
\begin{equation}\label{e:26}
\wh{p}^{Y}(t, |x|, |y|; \sigma_{\{0\}}<t)=\wh{p}^{Y}(t, |x|, -|y|)=\frac{1}{2}\left(\wh{p}^{\wh{Y}}(t, |x|, |y|)-\wh{p}^Y_{\IR_+}(t, |x|, |y|)\right).
\end{equation}
\end{description}
\end{proposition}
\begin{proof}
Since $\wh{Y}=|Y|$, for any $x, y\in E$,  any $\delta>0$ such that $[|y|-\delta, |y|+\delta]\subset (0, +\infty)$, it holds
\begin{equation*}
\IP_{|x|}\left(Y_t\in [|y|-\delta, |y|+\delta]\right)+\IP_{|x|}\left(Y_t\in [-|y|-\delta, -|y|+\delta]\right)=\IP_{|x|}\left(\wh{Y}_t\in [|y|-\delta, |y|+\delta]\right).
\end{equation*} 
Therefore, 
\begin{align*}
\IP_{|x|}\left(Y_t\in [|y|-\delta, |y|+\delta]; \sigma_{\{0\}}>t\right)&+\IP_{|x|}\left(Y_t\in [|y|-\delta, |y|+\delta]; \sigma_{\{0\}}\le t\right)
\\
&+\IP_{|x|}\left(Y_t\in [-|y|-\delta, -|y|+\delta]\right)=\IP_{|x|}\left(\wh{Y}_t\in [|y|-\delta, |y|+\delta]\right).
\end{align*}
This justifies \eqref{e:4.9}.  To justify \eqref{e:26}, observing that $Y$ is  strongly Markov and symmetric about $0$, by Proposition \ref{strongMarkov} we have 
\begin{align}
\wh{p}^{Y}(t, |x|, |y|; \sigma_{\{0\}}\le t)&=\int_0^t \IP_{|x|}\left(\sigma^Y_{\{0\}}\in ds\right)\wh{p}^{Y}(t-s, 0, |y|)\nonumber
\\
&=\int_0^t \IP_{|x|}\left(\sigma^Y_{\{0\}}\in ds\right)\wh{p}^{Y}(t-s, 0, -|y|) =\wh{p}^{Y}(t, |x|, -|y|).\label{e:25}
\end{align}
Now \eqref{e:26} readily  follows  from   \eqref{e:4.9}.
\end{proof}
\qed
In the following, we divide our discussion into four cases  depending on the locations of $x, y$. Note that due to the symmetry of $M$ with respect to $m_\gamma$, they essentially cover all the cases for $x,y\in E$.
\begin{description}
\item{Case (i)}:  $x, y\in  E_1\backslash \{\b0\}$;
\item{Case (ii)}: $x\in E_1\backslash \{\b0\}$, $y\in E_2\backslash \{\b0\}$;
\item{Case (iii)}: $x, y\in E_2\backslash \{\b0\}$;
\item{Case (iv)}: $x=\b0$, $y\in E$.  
\end{description}

\subsection{Case (i): both  $x, y \in E_1$}
We state the result for this case as the following proposition
\begin{proposition}
\begin{equation*}
p(t,x,y)=q(t,x, y)+\frac{1}{2}\left(p^{\wh{Y}}(t, |x|, |y|)-p^Y_{\IR_+}(t, |x|, |y|)\right), \quad  x,y\in E_1\backslash \{\b0\}.
\end{equation*}
\end{proposition}
\begin{proof}
For this case, we recall   that the density of $M^{E_1\backslash \{\b0\}}$,  $q(t,x, y)$,   has been computed in Proposition \ref{HK-killed-dBM}.  Denote by $\wh{p}^M$ the transition density of $M$ with respect to Lebesgue measure. We first notice that for $y\in E_1\backslash \{\b0\}$, $\wh{p}^M(t, 
\b0, y)$ is rotationally invariant in $y$.  Therefore, for any $y\in E_1\backslash \{\b0\}$, we may define
\begin{equation}\label{e:4.12}
\bar{p}^M(t, \b0, r):=\wh{p}^M(t,\b0, y), \quad \text{for } r=|y|. 
\end{equation}
Using this notation and polar coordinates, we have for any pair of  $a>b>0$, 
\begin{align*}
\int_a^b \wh{p}^{Y}(t,0,r)dr&=  \IP^{Y}_0( a\leq Y_t\leq b) = \IP^M_{\b0} ( M_t \in E_1 \hbox{ with }  a\leq |M_t |\leq b) \\
&=  \int_{\{y\in {E_1}:  a \leq |y| \leq b\}}\wh{p}^M (t,\b0,y)dy 
\\
&=\int_{a}^b 4\pi r^2\bar{p}^M(t,\b0,r)dr.
\end{align*}
This implies that 
\begin{equation}\label{e:4.13}
\wh{p}^{Y}(t,0,r)=4\pi r^2\bar{p}^M(t,\b0,r) \stackrel{\eqref{e:4.12}}{=}4\pi |y|^2 \wh{p}^M(t,\b0, y), \quad \text{for all }y\in E_1\backslash \{\b0\} \text{ and } r=|y|.
\end{equation}
Now on account of Proposition \ref{strongMarkov}, we have for $x,y\in E_1\backslash \{\b0\}$ that
\begin{align}
p(t,x,y)&= q(t,x,y)  +p(t,x,y;\, \sigma_{\{0\}}\le t)
\\
&=q(t,x, y)+\int_{0}^t \IP_x \left(\sigma^M_{\{\b0\}}\in ds\right)p^M(t-s, \b0, y) \nonumber
\\
&=q(t,x, y)+\int_{0}^t \IP_x \left(\sigma^M_{\{\b0\}}\in ds\right)\wh{p}^M(t-s,\b 0, y)\cdot \frac{2\pi |y|^2}{\gamma}\;e^{2\gamma |y|}  \nonumber
\\
&\stackrel{\eqref{e:4.13}}{=}q(t,x, y)+\int_{0}^t \IP_{|x|} \left(\sigma^Y_{\{0\}}\in ds\right)\wh{p}^{Y}(t-s, 0, |y|)\frac{1}{4\pi |y|^2}\frac{2\pi |y|^2}{\gamma}\;e^{2\gamma |y|} \nonumber
\\
&=q(t,x, y)+\int_{0}^t \IP_{|x|} \left(\sigma^Y_{\{0\}}\in ds\right)\wh{p}^{Y}(t-s, 0, |y|)\frac{1}{2\gamma}\;e^{2\gamma |y|}  \nonumber
\\
&=q(t,x, y)+\frac{1}{2\gamma}\;e^{2\gamma |y|}\wh{p}^Y\left(t, |x|, |y|; \sigma_{\{0\}}\le t\right).  \label{e:24}
\end{align}
Applying \eqref{e:26} to the right hand side of \eqref{e:24}, we have for $x,y\in E_1\backslash \{\b0\}$,
\begin{equation}\label{e:4.21}
p(t,x,y)=q(t,x, y)+\frac{1}{4\gamma}\;e^{2\gamma |y|}\left(\wh{p}^{\wh{Y}}(t, |x|, |y|)-\wh{p}^Y_{\IR_+}(t, |x|, |y|)\right).
\end{equation} 
Note that  
\begin{equation}\label{change-measure-wh-y-E2}
p^{\wh{Y}}(t,x,y)=\wh{p}^{\wh{Y}}(t,x,y)\frac{1}{\phi_\gamma(y)^2}=\wh{p}^{\wh{Y}}(t,x,y)\frac{1}{2\gamma}e^{2\gamma |y|}, \quad x, y\in \IR_+
\end{equation}
and 
\begin{equation}\label{change-measure-y-E2}
p^Y_{\IR_+}(t,x,y)=\wh{p}^Y_{\IR_+}(t,x,y)\frac{1}{\phi_\gamma(y)^2}=\wh{p}^Y_{\IR_+}(t,x,y)\frac{1}{2\gamma}e^{2\gamma |y|}, \quad x,y\in \IR_+,
\end{equation}
respectively. Replacing the second term on the right hand side of \eqref{e:4.21} with \eqref{change-measure-wh-y-E2} and \eqref{change-measure-y-E2} yields
\begin{equation*}
p(t,x,y)=q(t,x, y)+\frac{1}{2}\left(p^{\wh{Y}}(t, |x|, |y|)-p^Y_{\IR_+}(t, |x|, |y|)\right), \quad x,y\in E_1\backslash \{\b0\},
\end{equation*}
where  $p^{\wh{Y}}(t, x, y)$ and $p^Y_{\IR_+}(t,x,y)$ are given in \eqref{density-RBM-drift}  \eqref{density-pY_R-}, respectively.
\end{proof}
\qed

\subsection{Case (ii): $x\in  E_1\backslash \{\b0\}, \,y \in E_2\backslash \{\b0\}$}

\begin{proposition}\label{HKE-case2}
\begin{equation*}
p(t,x,y)=\frac{1}{2}\left(p^{\wh{Y}}(t, |x|, |y|)-p^Y_{\IR_+}(t, |x|, |y|)\right), \quad x\in E_1\backslash \{\b0\}, y\in E_2\backslash \{\b0\},
\end{equation*}
\end{proposition}
\begin{proof}
For this case, we notice that any path of $M$ has to pass $\b0$ in order to travel from $x$ to $y$.  It therefore follows from Proposition \ref{strongMarkov} that
\begin{align*}
\wh{p}(t,x,y)&= \wh{p}(t,x,y;\; \sigma^M_{\{0\}}\le t)=\int_{0}^t \IP_x \left(\sigma^M_{\{\b0\}}\in ds\right)\wh{p}^M(t-s, \b0, y)
\\
&=\int_{0}^t \IP_{|x|} \left(\sigma^Y_{\{0\}}\in ds\right)\wh{p}^Y(t-s, 0, -|y|)
\\
&=\wh{p}^Y(t, |x|, -|y|)\stackrel{\eqref{e:26}}{=}\frac{1}{2}\left(\wh{p}^{\wh{Y}}(t, |x|, |y|)-\wh{p}^Y_{\IR_+}(t, |x|, |y|)\right).
\end{align*}
By a change of measure in \eqref{change-measure-wh-y-E2} and \eqref{change-measure-y-E2}, it immediately follows
\begin{equation}
p(t,x,y)=\frac{1}{2}\left(p^{\wh{Y}}(t, |x|, |y|)-p^Y_{\IR_+}(t, |x|, |y|)\right).
\end{equation}
\end{proof}
\qed

\subsection{Case (iii): both $x, y \in E_2\backslash \{\b0\}$}

\begin{proposition}\label{HKE-case3}
\begin{equation*}
p(t,x,y)=\frac{1}{2}\left(p_{E_2\backslash \{\b0\}}(t,x,y)+p^{\wh{Y}}(t, |x|, |y|)\right), \quad x,y\in E_2\backslash \{\b0\};
\end{equation*}
\end{proposition}
\begin{proof}
 We note that similar to the previous two cases, it holds
\begin{align}
p(t,x,y)&=p_{E_2\backslash \{\b0\}}(t,x,y)+p^M\left(t, x, y; \sigma^M_{\{\b0\}}\le t\right)   \nonumber
\\
&=p_{E_2\backslash \{\b0\}}(t,x,y)+p^Y\left(t, -|x|, -|y|; \sigma^Y_{\{0\}}\le t\right)\nonumber
\\
&=p_{E_2\backslash \{\b0\}}(t,x,y)+\frac{1}{2\gamma}e^{2\gamma |y|}\;\wh{p}^Y(t, -|x|, -|y|; \sigma^Y_{\{0\}}\le t). \label{e:4.15}
\end{align}
Again we notice that $Y$ is symmetric about $0$, so
\begin{align}\label{e:4.16}
\wh{p}^Y(t, -|x|, -|y|; \sigma^Y_{\{0\}}\le t)=\wh{p}^Y(t, |x|, |y|; \sigma^Y_{\{0\}}\le t)\stackrel{\eqref{e:26}}{=}\frac{1}{2}\left(\wh{p}^{\wh{Y}}(t, |x|, |y|)-\wh{p}^Y_{\IR_+}(t, |x|, |y|)\right).
\end{align}
Applying \eqref{e:4.16} to the right hand side of \eqref{e:4.15} yields
\begin{align*}
p(t,x,y)&=p_{E_2\backslash \{\b0\}}(t,x,y)+\frac{1}{4\gamma}e^{2\gamma |y|}\left(\wh{p}^{\wh{Y}}(t, |x|, |y|)-\wh{p}^Y_{\IR_+}(t, |x|, |y|)\right)
\\
&=p_{E_2\backslash \{\b0\}}(t,x,y)+\frac{1}{2}\left(\wh{p}^{\wh{Y}}(t, |x|, |y|)-\wh{p}^Y_{\IR_+}(t, |x|, |y|)\right)\frac{1}{\phi(y)^2}
\\
&=p_{E_2\backslash \{\b0\}}(t,x,y)+\frac{1}{2}\left(p^{\wh{Y}}(t, |x|, |y|)-p^Y_{\IR_+}(t, |x|, |y|)\right)
\\
&=\frac{1}{2}\left(p^{\wh{Y}}(t, |x|, |y|)+p_{E_2\backslash \{\b0\}}(t,x,y)\right),
\end{align*}
where the last $``="$ is due to \eqref{density-pY_R-}.
\end{proof}
\qed

\subsection{Case  (iv): $x=\b0,\, y\in E$ }
The remaining case is that at least one of $x,y$ is $\b0$. In view of the symmetry of $M$ with respect to $m_\gamma$, without loss of generality, we assume $x=\b0$. 
\begin{proposition}
\begin{equation*}
p(t,\b0,y)=\frac{1}{2}\left(1+\frac{1}{\pi \gamma} e^{\gamma |y|-\gamma t^2/2} \int_0^\infty \frac{s e^{-s^2t/2}}{s^2+\gamma^2} \left(s^2\cos(s|y|)-s\gamma \sin(s|y|)\right)ds \right),\quad \text{for all }y\in E.
\end{equation*}
\end{proposition}
\begin{proof} We divide our discussion into two subcases depending on the position of $y$. 
\\
{\it Subcase (i)}: $y\in E_2\backslash \{\b0\}$. Recall that we let $\wh{p}^M$ be the transition density of $M$ with respect to Lebesgue measure on $E$.   Given any $x\in E_2$, any $0<a<b$, 
\begin{eqnarray*}
&&\int_{a<|u|<b, u\in E_2} \wh{p}^M(t, x, u)du=\int_{a<|u|<b, u\in E_2} p^M(t, x, u)m_\gamma (du) 
\\
&=& \IP_{x}\left[M_t\in E_2, \, a<|M_t|<b\right] = \IP_{-|x|}\left[-b<Y<-a\right]
\\
&=&\int_{-b}^{-a}p^Y(t, -|x|, \xi)\wt{\mathfrak{m}}(d\xi) = \int_{-b}^{-a}p^Y(t, -|x|, \xi )2\gamma e^{-2\gamma |\xi|}d\xi
\\
&\stackrel{v=-\xi}{=}&\int_a^b \wh{p}^Y(t, -|x|, -v)dv.
\end{eqnarray*}
Since $E_2\cong\IR_+$, this implies 
\begin{equation}\label{eq:4.25}
\wh{p}^M(t, x, u)=\wh{p}^Y(t, -|x|,- |u|), \quad \text{for all }t>0, x\in E_2, u\in E_2\backslash \{\b0\}. 
\end{equation}
Thus in view of the definitions of $m_\gamma$ and $\wt{\mathfrak{m}}$, it holds for all $x, u\in E_2$ that
\begin{equation}\label{eq:4.26}
p^M(t, x, u)=\wh{p}^M(t, x, u)\frac{1}{2\gamma} e^{2\gamma |u|}=\wh{p}^Y(t, -|x|, -|u|)\frac{1}{2\gamma}e^{2\gamma |u|}=p^Y(t, -|x|, -|u|)=p^Y(t,|x|, |u|),
\end{equation}
where the last $``="$ is due to the symmetry of $Y$: $Y\stackrel{d}{=}-Y$. 
It has been noted in  Corollary \ref{joint-cont-Y}, $p^Y$ is jointly continuous on $(0, +\infty)\times \IR\times \IR$.  We take a sequence of $\{x_n\}_{n\ge 1}\subset E_2\backslash \{\b0\}$ such that $|x_n|\downdownarrows 0$ as $n\rightarrow \infty$. By the joint continuity of $p^Y$,  
\begin{equation}\label{eq:4.27}
p^M(t, \b0, y)\stackrel{\eqref{eq:4.26}}{=}p^Y(t, 0, -|y|)=\lim_{n\rightarrow \infty} p^Y(t, -|x_n|, -|y|) \stackrel{\eqref{eq:4.26}}{=} \lim_{n\rightarrow \infty}p^M(t, x_n, y). 
\end{equation}
Now by Proposition \ref{HKE-case3}, 
\begin{eqnarray*}
&& p^M(t, \b0, y) \stackrel{\eqref{eq:4.27}}{=} \lim_{n\rightarrow \infty}p^M(t, x_n, y)
 \\
 &=& \lim_{n\rightarrow \infty} \frac{1}{2} \left[    p^{\wh{Y}}(t, |x_n|, |y|)-p^Y_{\IR_+}(t, |x_n|, |y|)\right]
\\
&=&\lim_{n\rightarrow \infty}\frac{1}{2}\Bigg[ 1 +\frac{1}{\pi \gamma}e^{\gamma (|x_n|+|y|)-\frac{\gamma^2 t}{2}}\int_{0}^\infty \frac{e^{-\frac{s^2t}{2}}}{s^2+\gamma^2}\left[s \cos(s|x_n|)-\gamma \sin(s|x_n|)\right]\left[s\cos(s|y|)-\gamma \sin (s|y|)\right]ds  
\\
&+&\frac{1}{\gamma\sqrt{8\pi t}}e^{-\gamma^2 t/2+\gamma(|x_n|+|y|)}\left(e^{-(|x_n|-|y|)^2/(2t)}-e^{-(|x_n|+|y|)^2/(2t)}\right)\Bigg]
\\
&=& \frac{1}{2}\left[1+\frac{1}{\pi\gamma}e^{\gamma |y|-\frac{\gamma^2t}{2}}  \left(s^2\cos(s|y|)-s\gamma \sin(s|y|)\right)ds \right], \quad y\in E_2\backslash \{\b0\}.
\end{eqnarray*}
{\it Subcase (ii)}: $y\in E_1\backslash \{\b0\}$. By the rotational invariance of $M$ on $E_1$, we set for $x\in E_1$ that
\begin{equation}
\bar{p}^M(t, x, r):=\wh{p}^M(t,x, y), \quad \text{for } r=|y|. 
\end{equation}
 Similar to subcase (i), given any $x\in E_2$, any $0<a<b$, 
\begin{eqnarray*}
&&\int_a^b 4\pi r^2 \bar{p}^M(t, x, r)  dr =\int_{a<|y|<b, u\in E_1} \wh{p}^M(t, x, y)dy
\\
&=&\int_{a<|y|<b, y\in E_1} p^M(t, x, y)m_\gamma (dy) 
\\
&=& \IP_{x}\left[M_t\in E_1, \, a<|M_t|<b\right] = \IP_{-|x|}\left[a<Y<b\right]
\\
&=&\int_a^b p^Y(t, -|x|, \xi)\wt{\mathfrak{m}}(du) = \int_a^b \wh{p}^Y(t, -|x|, \xi) d\xi.
\end{eqnarray*}
This justifies that
\begin{equation*}
4\pi |y|^2 \bar{p}^M(t, x, |y|) =\wh{p}^Y(t, -|x|, |y|) , \quad x\in E_2,\,y \in E_1\backslash \{\b0\}. 
\end{equation*}
Therefore, 
\begin{align}
p^M(t, x, y) & = \wh{p}^M(t,x, y)\frac{2\pi}{\gamma}|y|^2e^{2\gamma |y|}=\bar{p}^M(t, x, |y|) \frac{2\pi}{\gamma}|y|^2e^{2\gamma |y|}  \nonumber
\\
&=\wh{p}^Y(t, -|x|, |y|) \frac{1}{2\gamma }e^{2\gamma |y|} = p^{Y}(t, -|x|, |y|) , \quad \text{for }x\in E_2,\, y\in E_1\backslash \{\b0\}.  \label{eq:4.29}
\end{align}
Now again in view of the joint continuity of $p^Y$, taking a sequence of $\{x_n\}_{n\ge 1}\subset E_2\backslash \{\b0\}$ such that $|x_n|\downdownarrows 0$ as $n\rightarrow \infty$ yields that 
\begin{align}\label{eq:4.30}
p^M(t, \b0, y )\stackrel{\eqref{eq:4.29}}{=} p^Y(t, 0, |y|) = \lim_{n\rightarrow \infty} p^Y(t, -|x_n|, |y|)\stackrel{\eqref{eq:4.29}}{=} p^M(t, x_n, y)
\end{align}
It now follows from Proposition \ref{HKE-case2}  as well as the fact that $M$ is symmetric with respect to $m_\gamma$  that 
\begin{eqnarray*}
&& p^M(t, \b0, y) \stackrel{\eqref{eq:4.30}}{=} \lim_{n\rightarrow \infty}p^M(t, x_n, y) =p^M(t, y, x_n)
 \\
 &=& \lim_{n\rightarrow \infty} \frac{1}{2} \left[    p^{\wh{Y}}(t, |y|, |x_n|)-p^Y_{\IR_+}(t, |y|, |x_n|)\right]
\\
&=&\lim_{n\rightarrow \infty}\frac{1}{2}\Bigg[ 1 +\frac{1}{\pi \gamma}e^{\gamma (|x_n|+|y|)-\frac{\gamma^2 t}{2}}\int_{0}^\infty \frac{e^{-\frac{s^2t}{2}}}{s^2+\gamma^2}\left[s \cos(s|y|)-\gamma \sin(s|y|)\right]\left[s\cos(s|x_n|)-\gamma \sin (s|x_n|)\right]ds  
\\
&+&\frac{1}{\gamma\sqrt{8\pi t}}e^{-\gamma^2 t/2+\gamma(|x_n|+|y|)}\left(e^{-(|x_n|-|y|)^2/(2t)}-e^{-(|x_n|+|y|)^2/(2t)}\right)\Bigg]
\\
&=& \frac{1}{2}\left[1+\frac{1}{\pi\gamma}e^{\gamma |y|-\frac{\gamma^2t}{2}}  \left(s^2\cos(s|y|)-s\gamma \sin(s|y|)\right)ds \right], \quad y\in E_1\backslash \{\b0\}.
\end{eqnarray*}
Finally, since $y=\b0$ is a singleton of measure $0$, the proof is complete. 
\end{proof}
\qed

Theorem \ref{main-thm} is the combination of Proposition 10-13, which is now proved. 

\bibliographystyle{abbrv}
\bibliography{VaryDim}


\end{document}